\documentclass[a4paper]{article}

\usepackage[T1]{fontenc}
\usepackage[utf8]{inputenc}
\usepackage[pass]{geometry}
\usepackage{tikz}
	\usetikzlibrary{positioning}
\usepackage{amsmath}
\usepackage{amsfonts}
\usepackage{amssymb}
\usepackage{mathtools}
\usepackage{amsthm}
\usepackage{etoolbox}
\usepackage{graphicx}
\usepackage{caption}
\RequirePackage{hyperref}
	\hypersetup{linktoc=all,hidelinks,unicode=true}
\usepackage[capitalize,nameinlink]{cleveref}
\usepackage[noadjust]{cite}

\theoremstyle{definition}
\newtheorem{theorem}{Theorem}[section]
\newtheorem{lemma}[theorem]{Lemma}
\newtheorem{proposition}[theorem]{Proposition}
\newtheorem{corollary}[theorem]{Corollary}
\newtheorem{conjecture}{Conjecture}
\newtheorem{claim}{Claim}
\newtheorem{case}{Case}[theorem]
\newtheorem{subcase}{Case}[case]
\newtheorem{subsubcase}{Case}[subcase]
\newtheorem{example}{Example}[section]
\theoremstyle{remark}
\newtheorem{remark}{Remark}[section]
\crefname{claim}{Claim}{Claims}
\crefname{case}{Case}{Cases}
\crefname{subcase}{Case}{Cases}
\crefname{subsubcase}{Case}{Cases}
\crefname{conjecture}{Conjecture}{Conjectures}
\AfterEndEnvironment{case}{\noindent\ignorespaces}
\AfterEndEnvironment{subcase}{\noindent\ignorespaces}
\AfterEndEnvironment{subsubcase}{\noindent\ignorespaces}

\let\eqref\labelcref
\crefname{equation}{}{}

\crefname{enumi}{}{}

\renewcommand{\overrightarrow}[1]{\smash{\vec{#1}}}
\newlength{\revdirraise}
\newlength{\revdirextraraise}
\renewcommand{\overleftarrow}[1]{\smash{%
	\settoheight{\revdirraise}{$#1$}%
	\settodepth{\revdirextraraise}{$#1$}%
	\ooalign{$#1$\cr{\raisebox{\revdirraise+\revdirextraraise+0.85pt}%
		{\rlap{$\mkern4.1mu$\rotatebox[origin=c]{180}%
			{$\vec{\phantom #1}$}}}}}}}
\DeclareMathOperator{\diam}{diam}

\DeclarePairedDelimiter{\card}{\lvert}{\rvert}
\DeclarePairedDelimiter{\set}{\lbrace}{\rbrace}
\DeclarePairedDelimiterX{\defset}[2]{\lbrace}{\rbrace}
	{\,#1:#2\,}
\DeclarePairedDelimiter{\floor}{\lfloor}{\rfloor}

\AtBeginDocument{\frenchspacing}

\newcommand{\internalvertex}[2]{node [circle, inner sep=0pt, outer sep=0pt, minimum size=4pt, fill=#1, solid, draw=#2] {}}
\newcommand{\vertex}{\internalvertex{black}{black}}
\newcommand{\minivertex}{node [circle, inner sep=0pt, outer sep=0pt, minimum size=3pt, fill=black, draw=black] {}}
\newcommand{\point}{\node [inner sep=0pt,outer sep=0pt,minimum size=0pt]}

\newcommand{\figFparts}{%
\begin{tabular}{c@{\kern2em}c}
\figTone&\figTtwo\\[2em]
\figTthree&\figTfour
\end{tabular}}

\newcommand{\figTone}{%
\begin{tikzpicture}[y=0.87cm]
\point (a) at (-1,0) {};
\point (b) at (2,0) {};
\figToneblock
\draw (a) \vertex (b) \vertex;
\draw[overlay] (-1.15,0.85) node {$T_1$:};
\draw[overlay] (a) node [anchor=south] {$a$};
\draw[overlay] (b) node [anchor=south] {$b$};
\end{tikzpicture}}

\newcommand{\figTtwo}{%
\begin{tikzpicture}[y=0.87cm]
\point (a) at (-2,0) {};
\point (b) at (3,0) {};
\figTtwoblock
\draw (a) \vertex (b) \vertex;
\draw[overlay] (-2.15,0.85) node {$T_2$:};
\draw[overlay] (a) node [anchor=south] {$a$};
\draw[overlay] (b) node [anchor=south] {$b$};
\end{tikzpicture}}

\newcommand{\figTthree}{%
\begin{tikzpicture}[y=0.87cm]
\point (a) at (-1.5,0) {};
\point (b) at (2.5,0) {};
\figTthreeblock
\draw (a) \vertex (b) \vertex;
\draw[overlay] (-1.65,0.85) node {$T_3$:};
\draw[overlay] (a) node [anchor=south] {$a$};
\draw[overlay] (b) node [anchor=south] {$b$};
\end{tikzpicture}}

\newcommand{\figTfour}{%
\begin{tikzpicture}[y=0.87cm]
\point (a) at (-2,0) {};
\point (b) at (3,0) {};
\figTfourblock
\draw (a) \vertex (b) \vertex;
\draw[overlay] (-2.15,0.85) node {$T_4$:};
\draw[overlay] (a) node [anchor=south] {$a$};
\draw[overlay] (b) node [anchor=south] {$b$};
\end{tikzpicture}}

\newcommand{\figFfourtwo}{%
\begin{tabular}{c@{\kern1em}c@{\kern1em}c@{\kern1em}c@{\kern1em}c@{\kern1em}c}
\figF11&\figF13&\figF12&\figF14&\figF44\\[1em]
\figF33&\figF32&\figF34&\figF22&\figF24
\end{tabular}}

\newcommand{\figF}[2]{%
\bgroup\let\vertex\minivertex
\begin{tikzpicture}[scale=0.42,y=0.87cm]
\point (a) at ({-0.2-\numberstosize{#1}{#2}/2},0) {};
\point (b) at ({1.2+\numberstosize{#1}{#2}/2},0) {};
\begin{scope}[shift={(0,1)}]
\numbertoTblock{#1}
\end{scope}
\begin{scope}[shift={(0,-1)}]
\numbertoTblock{#2}
\end{scope}
\draw (a) \vertex (b) \vertex;
\end{tikzpicture}
\egroup}

\newcommand{\numbertoTblock}[1]{%
	\ifnum #1=1 \figToneblock   \fi
	\ifnum #1=2 \figTtwoblock   \fi
	\ifnum #1=3 \figTthreeblock \fi
	\ifnum #1=4 \figTfourblock  \fi
}
\newcommand{\numbertosize}[1]{%
	\ifnum #1=1 1 \fi
	\ifnum #1=2 3 \fi
	\ifnum #1=3 2 \fi
	\ifnum #1=4 3 \fi
}
\newcommand{\numberstosize}[2]{%
	\ifnum \numbertosize{#1}>\numbertosize{#2}
		\numbertosize{#1}
	\else
		\numbertosize{#2}
	\fi
}

\newcommand{\figToneblock}{%
\begin{scope}[shift={(-1,0)}]
\foreach \y in {-0.5,0.5}{
	\draw (a) -- (1,\y) (2,\y) -- (b);
	\foreach \z in {-0.5,0.5}{
		\draw (1,\y) -- (2,\z);}}
\foreach \x in {1,2}{
	\draw (\x,-0.5) -- (\x,0.5);}
\foreach \x in {1,2}{
	\foreach \y in {-0.5,0.5}{
		\draw (\x,\y) \vertex;}}
\end{scope}}

\newcommand{\figTtwoblock}{%
\begin{scope}[shift={(-2,0)}]
\foreach \y in {-0.5,0.5}{
	\draw (a) -- (1,\y) (4,\y) -- (b);
	\draw (2,\y) -- (3,\y);
	\foreach \x in {1,3}{
		\foreach \z in {-0.5,0.5}{
			\draw (\x,\y) -- (\x+1,\z);}}}
\foreach \x in {1,...,4}{
	\draw (\x,-0.5) -- (\x,0.5);}
\foreach \x in {1,...,4}{
	\foreach \y in {-0.5,0.5}{
		\draw (\x,\y) \vertex;}}
\end{scope}}

\newcommand{\figTthreeblock}{%
\begin{scope}[shift={(-1.5,0)}]
\foreach \y in {-0.5,0.5}{
	\draw (a) -- (1,\y) (3,\y) -- (b);
	\foreach \x in {1,...,2}{
		\foreach \z in {-0.5,0.5}{
			\draw (\x,\y) -- (\x+1,\z);}}}
\foreach \x in {1,3}{
	\draw (\x,-0.5) -- (\x,0.5);}
\foreach \x in {1,...,3}{
	\foreach \y in {-0.5,0.5}{
		\draw (\x,\y) \vertex;}}
\end{scope}}

\newcommand{\figTfourblock}{%
\begin{scope}[shift={(-2,0)}]
\foreach \y in {-0.5,0.5}{
	\draw (a) -- (1,\y) (4,\y) -- (b);
	\foreach \x in {1,...,3}{
		\foreach \z in {-0.5,0.5}{
			\draw (\x,\y) -- (\x+1,\z);}}}
\foreach \x in {1,4}{
	\draw (\x,-0.5) -- (\x,0.5);}
\foreach \x in {1,...,4}{
	\foreach \y in {-0.5,0.5}{
		\draw (\x,\y) \vertex;}}
\end{scope}}

\newcommand{\figclawfreeexample}{%
\begin{tikzpicture}
\begin{scope}[shift=(210:2)] % V₁
	\begin{scope}[shift=(120:1.3)]
		\point (v1-1) at (-30:1.5) {};
		\point (v1-2) at (-50:1.5) {};
		\point (v1-3) at (-70:1.5) {};
		\point (v1-4) at (-90:1.5) {};
	\end{scope}
\end{scope}
\begin{scope}[shift=(150:2)] % V₂
	\begin{scope}[shift=(-120:1.3)]
		\point (v2-1) at (30:1.5) {};
		\point (v2-2) at (50:1.5) {};
		\point (v2-3) at (70:1.5) {};
		\point (v2-4) at (90:1.5) {};
	\end{scope}
\end{scope}
\begin{scope}[shift=(90:2)] % V₃
	\point (v3-1) at (0,0.3) {};
	\point (v3-2) at (0,-0.3) {};
\end{scope}
\begin{scope}[shift=(30:2)] % V₄
	\begin{scope}[shift=(-60:1.3)]
		\point (v4-1) at ( 90:1.5) {};
		\point (v4-2) at (110:1.5) {};
		\point (v4-3) at (130:1.5) {};
		\point (v4-4) at (150:1.5) {};
	\end{scope}
\end{scope}
\begin{scope}[shift=(-30:2)] % V₅
	\begin{scope}[shift=(60:1.35)]
		\point (v5-1) at (- 95:1.5) {};
		\point (v5-2) at (-120:1.5) {};
		\point (v5-3) at (-145:1.5) {};
	\end{scope}
\end{scope}
\begin{scope}[shift=(-90:2)] % V₆
	\point (v6-1) at (0,0.3) {};
	\point (v6-2) at (0,-0.3) {};
\end{scope}
\foreach \x in {1,2}{
	\foreach \y in {1,2,3,4}{
		\draw (v3-\x) -- (v4-\y);
		\draw (v6-\x) -- (v1-\y);}}
\foreach \x in {1,2,3,4}{
	\foreach \y in {1,2,3,4}
		\draw (v1-\x) -- (v2-\y);
	\foreach \y in {1,2}
		\draw (v2-\x) -- (v3-\y);
	\foreach \y in {1,2,3}
		\draw (v4-\x) -- (v5-\y);}
\foreach \x in {1,2,3}
	\foreach \y in {1,2}
		\draw (v5-\x) -- (v6-\y);
\draw (v1-1) -- (v1-2) -- (v1-3) -- (v1-4) -- (v1-1);
\draw (v2-1) -- (v2-2) -- (v2-3) -- (v2-4) -- (v2-1);
\foreach \x in {1,2,3}
	\foreach \y in {\x,...,4}
		\draw (v4-\x) -- (v4-\y);
\foreach \x in {1,2}
	\foreach \y in {\x,...,3}
		\draw (v5-\x) -- (v5-\y);
\draw (v6-1) -- (v6-2);
\foreach \x in {1,2,4}
	\foreach \y in {1,2,3,4}
		\draw (v\x-\y) \vertex;
\foreach \x in {3,6}
	\foreach \y in {1,2}
		\draw (v\x-\y) \vertex;
\foreach \y in {1,2,3}
	\draw (v5-\y) \vertex;
\draw [dashed,rotate=210] (2,0.02) ellipse (1 and 0.4);
\draw [dashed,rotate=150] (2,-0.02) ellipse (1 and 0.4);
\draw [dashed,rotate= 90] (2,0) ellipse (0.6 and 0.3);
\draw [dashed,rotate= 30] (2,0.02) ellipse (1 and 0.4);
\draw [dashed,rotate=-30] (2,-0.02) ellipse (0.9 and 0.4);
\draw [dashed,rotate=-90] (2,0) ellipse (0.6 and 0.3);
\node at (210:3)   [anchor=  30] {$V_1$};
\node at (150:3)   [anchor= -30] {$V_2$};
\node at ( 90:2.6) [anchor= -90] {$V_3$};
\node at ( 30:3)   [anchor=-150] {$V_4$};
\node at (-30:2.9) [anchor= 150] {$V_5$};
\node at (-90:2.6) [anchor=  90] {$V_6$};
\end{tikzpicture}}

\newcommand{\figcounterexample}{%
\begin{tikzpicture}
\foreach \x in {-1,1}{
	\foreach \y in {-1,0,1}{
		\begin{scope}[x=\x cm,shift={(0.8,2.2*\y)}]
			\point (v-\x-\y-1) at (120:0.4) {};
			\point (v-\x-\y-2) at (240:0.4) {};
			\point (u-\x-\y-1) at ( 30:1.3) {};
			\point (u-\x-\y-2) at ( 10:1.3) {};
			\point (u-\x-\y-3) at (-30:1.3) {};
			\point (u-\x-\y-4) at (-10:1.3) {};
			\point (w-\x-\y)   at (  0:2.1) {};
		\end{scope}}
	\begin{scope}[x=\x cm,shift={(4.88,0)}]
		\point (h-\x-1-1)  at (120:1.25) {};
		\point (h-\x-1-2)  at (144:1.25) {};
		\point (h-\x-0-1)  at (168:1.25) {};
		\point (h-\x-0-2)  at (192:1.25) {};
		\point (h-\x--1-1) at (216:1.25) {};
		\point (h-\x--1-2) at (240:1.25) {};
	\end{scope}}
\foreach \x in {-1,1}{
	\foreach \y in {-1,0,1}{
		\draw (w-\x-\y) -- (h-\x-\y-1) -- (h-\x-\y-2) -- (w-\x-\y);
		\foreach \z in {1,2,3,4}{
			\draw (w-\x-\y) -- (u-\x-\y-\z);
			\foreach \w in {\z,...,4}
				\draw (u-\x-\y-\z) -- (u-\x-\y-\w);
			\foreach \w in {1,2}
				\draw (u-\x-\y-\z) -- (v-\x-\y-\w);}
			\draw (v-\x-\y-1) -- (v-\x-\y-2);}
	\foreach \z in {1,2}{
		\foreach \w in {1,2}{
			\draw (h-\x-1-\z) -- (h-\x-0-\w);
			\draw (h-\x-1-\z) -- (h-\x--1-\w);
			\draw (h-\x-0-\z) -- (h-\x--1-\w);}}}
\foreach \y in {-1,0,1}
	\foreach \z in {1,2}
		\draw (v--1-\y-\z) -- (v-1-\y-\z);
\foreach \x in {-1,1}{
	\foreach \y in {-1,0,1}{
		\foreach \z in {1,2}{
			\draw (v-\x-\y-\z) \vertex;
			\draw (h-\x-\y-\z) \vertex;}
		\foreach \z in {1,2,3,4}
			\draw (u-\x-\y-\z) \vertex;
		\draw (w-\x-\y) \vertex;}}
\end{tikzpicture}}

\newcommand{\innerreduction}{%
	\point (x) at (-7,0) {};
	\point (y) at (-6,0) {};
	
	\draw (x) -- +(120:0.5);
	\draw (x) -- (aend);
	\draw (x) -- node [above] {$e$} (y);
	\draw (y) -- (bend);
	\draw (y) -- +(-60:0.5);
	\draw (x) \vertex node [left=3pt] {$x$};
	\draw (y) \vertex node [right=3pt] {$y$};
	
	\node at (-4.5,0) {$\rightarrow$};
	
	\point (we) at (0,0) {};
	\draw (-3,0)
		\foreach \x/\y in {1/30,2/10,3/-10,4/-30}
			{+(\y:1.5) node (u\x) [inner sep=0pt,outer sep=0pt] {}}
		+(-0.5,1.2) node (w1) [inner sep=0pt,outer sep=0pt] {}
		+(-0.5,-1.2) node (w2) [inner sep=0pt,outer sep=0pt] {};
	\draw (1.5,0)
		\foreach \x/\y in {1/150,2/210,3/90,4/30,5/330,6/270}
			{+(\y:0.5) node (v\x) [inner sep=0pt,outer sep=0pt] {}}
		+(60:1.5) node (w3) [inner sep=0pt,outer sep=0pt] {}
		+(300:1.5) node (w4) [inner sep=0pt,outer sep=0pt] {};
	
	\foreach \x in {1,...,4}{
		\foreach \y in {\x,...,4}
			\draw (u\x) -- (u\y);
		\foreach \y in {1,2,e}
			\draw (u\x) -- (w\y);}
	\foreach \x in {1,...,6}
		\foreach \y in {\x,...,6}
			\draw (v\x) -- (v\y);
	\foreach \x/\y in {1/e,2/e,3/3,4/3,5/4,6/4}
		\draw (v\x) -- (w\y);
	
	\draw (we) \vertex node [above] {$w_e$}
		(w1) \vertex
		(w2) \vertex
		(w3) \vertex
		(w4) \vertex;
	\foreach \x in {1,2,3,4}
		\draw (u\x) \vertex;
	\foreach \x in {1,2,3,4,5,6}
		\draw (v\x) \vertex;
	\draw [dashed] (1.5,0) circle (0.8) +(0,0.8) node [above] {$V_y$};
	\draw [dashed] (-1.7,0) ellipse (0.4 and 1) +(0,1) node [above] {$U_x$};
	}
\newcommand{\figreduction}{%
	\begin{tikzpicture}
	\path (-7,0) +(-120:0.5) node (aend) [inner sep=0pt,outer sep=0pt] {};
	\path (-6,0) +(60:0.5) node (bend) [inner sep=0pt,outer sep=0pt] {};
	\innerreduction
	\end{tikzpicture}}
\newcommand{\fighamreduction}{%
	\begin{tikzpicture}[densely dotted]
	\path (-7,0) +(-120:0.6) node (aend) [inner sep=0pt,outer sep=0pt] {};
	\path (-6,0) +(60:0.6) node (bend) [inner sep=0pt,outer sep=0pt] {};
	\innerreduction
	\draw [thick,solid]
		(-7,0) +(120:0.6) -- (-7,0) -- (-6,0) -- +(-60:0.6)
		(w1) -- (u1) -- (u2) -- (u3) -- (u4) -- (we)
		-- (v2) -- (v1) -- (v3) -- (w3) -- (v4) -- (v5) -- (v6) -- (w4);
	\path (-7,0)
		+(120:0.4) node [anchor=-120] {$d$}
		+(-120:0.4) node [anchor=120] {$a$}
		(-6,0)
		+(-60:0.4) node [anchor=60] {$f$}
		+(60:0.4) node [anchor=-60] {$b$}
		(w1) node [above] {$w_d$}
		(w2) node [below=1pt] {$w_a$}
		(w3) node [right] {$w_b$}
		(w4) node [right] {$w_f$};
	\end{tikzpicture}}

\hypersetup{%
	pdftitle={On Hamiltonicity of regular graphs with bounded second neighborhoods},
	pdfauthor={Armen S. Asratian, Jonas B. Granholm},
	pdfkeywords={Regular graphs, Hamilton cycles, locally connected graphs}}

\begin{document}

\title{On Hamiltonicity of regular graphs with bounded second neighborhoods}

\author{Armen S. Asratian%
	\footnote{Department of Mathematics, Linköping University, email: armen.asratian@liu.se},
	Jonas B. Granholm%
	\footnote{Department of Mathematics, Linköping University, email: jonas.granholm@liu.se}}
\date{}
\maketitle

%——————————————————————————————————
%	Abstract
%——————————————————————————————————

\begin{abstract}
\noindent
Let $\mathcal{G}(k)$ denote the set of
connected $k$-regular graphs~$G$, $k\geq 2$,
where the number of vertices at distance~$2$ from
any vertex in~$G$ does not exceed~$k$.
Asratian (2006) showed (using other terminology) that
a graph $G\in \mathcal{G}(k)$
is Hamiltonian if for each vertex~$u$ of~$G$ the subgraph induced by the set
of vertices at distance at most~2 from~$u$ is 2-connected.
We prove here that in fact all graphs in the sets $\mathcal{G}(3)$, $\mathcal{G}(4)$ and $\mathcal{G}(5)$
are Hamiltonian.
We also prove that the problem of determining whether there exists a Hamilton cycle in
a graph from $\mathcal{G}(6)$ is NP-complete.
Nevertheless we show that
every locally connected graph $G\in \mathcal{G}(k)$, $k\geq 6$,
is Hamiltonian
and that for each non-Hamiltonian cycle~$C$ in~$G$
there exists a cycle~$C'$ of length~$\card{V(C)}+\ell$ in~$G$, $\ell\in\set{1,2}$,
such that $V(C)\subset V(C')$.
Finally, we note that all our conditions for Hamiltonicity apply to infinitely many graphs with large diameters.
\end{abstract}

\bgroup
\par
\egroup

%——————————————————————————————————
%	Introduction
%——————————————————————————————————

\section{Introduction}

In this paper we continue our investigation of interconnections between local properties of a graph and
its Hamiltonicity (see, for example,
\cite{asratian06,asratian96,veldman96,asratian18a,asratian18b,asratyan85,hasratian90,asratian98}).

We use~\cite{diestel} for terminology and notation not defined here,
and consider finite undirected graphs without loops and multiple edges only.
A \emph{Hamilton cycle} of a graph~$G=(V(G),E(G))$ is a cycle containing every vertex of~$G$.
A graph that has a Hamilton cycle is called \emph{Hamiltonian}.
There is a vast literature in graph theory devoted to
obtaining sufficient conditions for Hamiltonicity
(see, for example, the surveys~\cite{gould03,gould14,hao13}).
Our paper is devoted to the investigation of Hamiltonicity of $k$-regular graphs, that is,
graphs where all vertices have the same degree~$k$.
It is known~\cite{picoul} that for any fixed $k\geq 3$, determining whether a $k$-regular graph has a Hamilton
cycle is an NP-complete problem.

Dirac~\cite{dirac52} proved that
a graph~$G$ on at least three vertices is Hamiltonian if the degree of every vertex of~$G$
is at least $\card{V(G)}/2$.
Nash-Williams~\cite{nash-williams71} showed that the result of Dirac can be relaxed for regular graphs as follows:
A $k$-regular graph~$G$ with $k\geq 2$
is Hamiltonian if $k\geq {|V(G)|-1\over 2}$.
Jackson~\cite{jackson80} showed that if $G$ is 2-connected then the bound $ {|V(G)|-1\over 2}$ can be replaced by ${1\over 3}|V(G)|$.
Better bounds were obtained for 3-connected $k$-regular graphs 
(see e.g. \cite{broersma96,kuhn,claw:li}).

All the above mentioned conditions for a regular graph~$G$
contain a global parameter of~$G$, namely the number of vertices,
and only apply to graphs with
large vertex degrees (\,$\geq \text{constant}\cdot|V(G)|$\,) and
small diameters (\,$o(|V(G)|)$\,).

Another type of sufficient conditions for Hamiltonicity of a graph~$G$,
which contain no global parameter of~$G$,
was obtained by Chartrand and Pippert~\cite{pippert74},
using the concept of local connectedness.
A graph~$G$ is called \emph{locally connected}
if for every vertex~$u$ of~$G$ the subgraph induced by the set of neighbors of~$u$ is connected.
The result of Chartrand and Pippert~\cite{pippert74} is devoted to graphs with maximum degree at most four
and for regular graphs implies the following:

\begin{proposition}
\label{pippert74}
All connected, locally connected, $k$-regular graphs are Hamiltonian if $k\leq 4$.
\end{proposition}
Locally connected 5-regular graphs were considered in~\cite{kikust75}:
\begin{theorem}[Kikust~\cite{kikust75}]
\label{kikust75}
Every connected, locally connected 5-regular graph~$G$ is Hamiltonian.
\end{theorem}

The situation is different for locally connected $k$-regular graphs with $k\geq 6$.
\begin{theorem}[Irzhavski~\cite{irzhav}]
\label{irzhav}
For any fixed $k\geq 6$,
determining whether a locally connected $k$-regular graph is Hamiltonian
is an NP-complete problem.
\end{theorem}
Some other results on locally connected graphs can be found in \cite{hendry,gordon11,vanaardt16,aardt}.

Oberly and Sumner~\cite{oberly79} showed that the concept of local connectedness is fruitful for Hamiltonicity of
\emph{claw-free graphs}, that is,
graphs that have no induced subgraph isomorphic to~$K_{1,3}$.

\begin{theorem}[Oberly and Sumner~\cite{oberly79}]
\label{oldthm:oberly}
Every connected, locally connected, claw-free graph on at least 3~vertices is Hamiltonian.
\end{theorem}

Other properties of claw-free graphs were found in
\cite{asratian96,asratian98,clark,faudree97,kaiser,matt,ryj97,wej}.
The results on locally connected graphs can be formulated in terms of balls.
For a vertex~$u$ of a graph~$G$, the \emph{ball of radius~$r$ centered at~$u$} is the subgraph of~$G$ induced by the set
$M_r(u)$ of vertices at distance at most~$r$ from~$u$.
Clearly, $G$~is locally connected if and only if every ball of radius~1 in~$G$ is 2-connected.

Some Hamiltonian properties of a graph have been obtained
using the structure of balls of radius~2 (see, for example, \cite{hasratian90,asratian98,asratian06,asratian18a,asratian18b,asratyan85,lai,ryj,veldman96}).
In particular, Asratian~\cite{asratian06} obtained the following result:

\begin{theorem}[Asratian~\cite{asratian06}]
\label{asratian06}
Let $G$ be a connected $k$-regular graph where
$k\geq {|M_2(u)|-1\over 2}$ for every vertex $u\in V(G)$ and every ball of radius two in~$G$ is 2-connected.
Then $G$ is Hamiltonian.
\end{theorem}

\cref{asratian06} is a generalization of the result of Nash-Williams~\cite{nash-williams71}
since a connected $k$-regular graph~$G$ with $k\geq {|V(G)|-1\over 2}$ is 2-connected and any
ball of radius~2 in~$G$ is the graph~$G$ itself.

Let $\mathcal{G}(k)$ denote the set of
connected $k$-regular graphs~$G$
where the number of vertices at distance~$2$ from
any vertex in~$G$ does not exceed~$k$.
Then one can show (see \cref{lemma1}) that \cref{asratian06} can be reformulated as follows%
\footnote{This formulation will be more convenient for our further descriptions and proofs.}:

\begin{theorem}
\label{asratian21}
Let $G$ be a graph in $\mathcal{G}(k)$ such that every ball of radius two in~$G$ is 2-connected.
Then $G$ is Hamiltonian.
\end{theorem}

The starting point of our present research was the following question:
Is it possible to omit the condition on 2-connectedness of balls of radius~2 in \cref{asratian21}
without losing the Hamiltonicity of~$G$? This question was motivated by the fact that the set $\mathcal{G}(k)$,
$k\geq 3$, contains infinitely many Hamiltonian graphs where no ball of radius~2 is 2-connected
(see \cref{claw-free1,claw-free2} in \cref{sec:defs}).

In the present paper we find new classes of Hamiltonian graphs
not satisfying known conditions for Hamiltonicity, e.g. the conditions of \cref{kikust75,oldthm:oberly}.
We prove that all graphs in the sets $\mathcal{G}(3)$, $\mathcal{G}(4)$ and $\mathcal{G}(5)$
are Hamiltonian,
that is, the condition of 2-connectedness of balls of radius~2 in {\cref{asratian21}} can be omitted if $k\leq 5$.
Moreover we characterize all graphs in
$\mathcal{G}(3)$, $\mathcal{G}(4)$ and $\mathcal{G}(5)$
where not all balls of radius~2 are 2-connected.
Our results on 3- and 4-regular graphs
imply \cref{pippert74}, and
our result on 5-regular graphs (\cref{five:regular}) and Kikust's theorem (\cref{kikust75}) are incomparable to each other in the sense that neither theorem implies the other.

We also show that in contrast with the sets $\mathcal{G}(3)$, $\mathcal{G}(4)$ and $\mathcal{G}(5)$,
the set $\mathcal{G}(k)$
contains non-Hamiltonian graphs for any $k\geq 6$.
Furthermore, we prove that the problem of determining whether there exists a Hamilton cycle in a graph from $\mathcal{G}(6)$ is NP-complete.
Despite this and \cref{irzhav},
we show that
every locally connected graph $G\in \mathcal{G}(k)$, $k\geq 6$, is Hamiltonian and
for any $k\ge 30$ the set $\mathcal{G}(k)$ contains
infinitely many locally connected graphs not satisfying the conditions of \cref{oldthm:oberly}.
We also prove that
if $C$ is a non-Hamiltonian cycle in a locally connected graph $G\in \mathcal{G}(k)$
then there exists a cycle~$C'$ of length~$\card{V(G)}+\ell$ in~$G$, $\ell\in\set{1,2}$,
such that $V(C)\subset V(C')$.
The paper is concluded with a conjecture.

%——————————————————————————————————
%	Definitions
%——————————————————————————————————

\section{Definitions and preliminary results}
\label{sec:defs}

The distance between vertices $u$ and~$v$ in~$G$
is denoted by $d_G(u,v)$ or simply $d(u,v)$.
The greatest distance between any two vertices in
a connected graph~$G$ is the \emph{diameter} of~$G$, denoted by $\diam(G)$.
For each vertex $u\in V(G)$ and integer $r\ge1$ we denote by $N_r(u)$ and $M_r(u)$
the set of all vertices $v\in V(G)$ with $d(u,v)=r$ and $d(u,v)\leq r$, respectively.
The set $N_1(u)$ is called the {\em neighborhood} of~$u$ and usually is denoted by $N(u)$. The number of vertices in $N(u)$ is the degree of~$u$, denoted by $d(u)$. The set $N_2(u)$ is called the {\em second neighborhood} of~$u$.
The \emph{ball of radius~$r$ centered at~$u$}, denoted by $G_r(u)$, is the subgraph of~$G$ induced by the set $M_r(u)$.

Let~$C$ be a cycle of a graph~$G$.
We denote by $\overrightarrow C$ the cycle~$C$ with a given orientation,
and by $\overleftarrow C$ the cycle~$C$ with the reverse orientation.
If $u,v\in V(C)$ then $u\overrightarrow Cv$ denotes the consecutive vertices of~$C$
from $u$ to~$v$ in the direction specified by $\overrightarrow C$.
The same vertices in reverse order are given by $v\overleftarrow Cu$.
We use $u^+$ to denote the successor of~$u$ on $\overrightarrow C$
and $u^-$ to denote its predecessor.
Analogous notation is used with respect to paths instead of cycles.

A graph is called {\em complete $r$-partite} ($r\geq 2$)
if its vertices can be partitioned into $r$~nonempty independent sets $V_1,\ldots,V_r$
such that two vertices are adjacent if and only if
they do not belong to the same set~$V_i$ for any~$i\in\set{1,\dotsc,r}$.
A complete $r$-partite graph
with independent sets $V_1,\ldots,V_r$ of sizes
$n_{1},\ldots,n_{r}$ is denoted by~$K_{n_{1},\ldots,n_{r}}$.

\begin{lemma}
\label{lemma1}
The following properties are equivalent for a connected $k$-regular graph~$G$, $k\geq 3$:
\begin{enumerate}
\item $G\in \mathcal{G}(k)$.
\item $k\geq {|M_2(u)|-1\over 2}$ for every vertex~$u$ of~$G$.
\item $d(u)+d(v)\geq |M_2(w)|-1$ for any induced path $uwv$ in~$G$.
\item $|N(u)\cap N(v)|\geq \card[\big]{M_2(w)\setminus \bigl(N(u)\cup N(v)\bigr)}-1$ for any induced path $uwv$.
\end{enumerate}
\end{lemma}
\begin{proof}
Let $G$ be a connected $k$-regular graph.
For any $u\in V(G)$, $|M_2(u)|=1+k+|N_2(u)|$. Therefore the condition
$k\geq {|M_2(u)|-1\over 2}$ is equivalent to the condition $|N_2(u)|\leq k$
which means that $G\in\mathcal{G}(k)$. Thus (i) is equivalent to (ii).
Evidently, (ii) is equivalent to (iii) because $d(u)=k=d(v)$.
Finally, (iii) is equivalent to (iv), since $d(u)+d(v)=|N(u)\cap N(v)|+|N(u)\cup N(v)|$.
\end{proof}

\begin{lemma}
\label{lem:cut-vertex-structure}
Let $G$ be a graph in~$\mathcal{G}(k)$.
If $v$ is a cut vertex of a ball of radius~2 in~$G$,
then $v$ is a cut vertex of the ball~$G_2(v)$,
the subgraph $G_2(v)-v$ has exactly two components with $k$~vertices each,
and every neighbor of~$v$ is adjacent to
all other vertices of the component of $G_2(v)-v$ it lies in.
\end{lemma}
\begin{proof}
If $v$~is a cut vertex of the ball~$G_2(v)$,
then the subgraph $G_2(v)-v$ has at least two components, at most $2k$ vertices,
and each neighbor of~$v$ has $k-1$ neighbors in this subgraph.
This is only possible if $G_2(v)-v$ has exactly two components with $k$~vertices each
and each neighbor of~$v$ is adjacent to
all other vertices of the component of $G_2(v)-v$ it lies in.

Now suppose that $v$~is a cut vertex of the ball~$G_2(u)$
for some vertex~$u\ne v$.
Then $uv\in E(G)$.
Let $F$ denote the component of $G_2(u)-v$ where $u$~lies,
and let $N(v)\cap V(F)=\{u_1,\dotsc,u_p\}$, where $u=u_1$ and $p<k$.
Since $G$~is $k$-regular,
$v$~has $k-p$ neighbors that do not lie in~$F$,
we will call these vertices $x_1,\dotsc,x_{k-p}$.
Furthermore, since $G$~is $k$-regular,
$u$~has at least $k-p$~neighbors in $G_2(u)-M_1(v)$,
and $x_1$ has at least $p$~neighbors outside~$G_2(u)$;
all of these vertices are at distance~2 from~$v$.
Since there are at most $k$~vertices at distance~2 from~$v$,
this means that $u$~is adjacent to the vertices~$v,u_2,\dotsc,u_p$
and exactly $k-p$~other vertices which we will call~$w_1,\dotsc,w_{k-p}$,
and the vertex~$x_1$ is adjacent to exactly $p$~vertices outside~$G_2(u)$
which we will call~$y_1,\dotsc,y_{p}$. Moreover, it is clear that the vertices $x_1,\dotsc,x_{k-p}$ are
adjacent to each other and all of them are adjacent to the vertices~$y_1,\dotsc,y_p$.
Also, since we now know that all neighbors of $u_2,\dotsc,u_p$ lie in~$G_2(u)$,
it is easy to see that the vertices $u_2,\dotsc,u_p$ are all adjacent to each other
and to $w_1,\dotsc,w_{k-p}$.
It remains to show that $G_2(v)-v$ is not connected,
that is, that $w_iy_j\notin E(G)$ for all $i,j$.
But this follows from the fact that $N(w_i)$~is contained in~$G_2(u)$ for all~$i$,
and we know that $y_j$~does not lie in~$G_2(u)$ for all~$j$.
The \lcnamecref{lem:cut-vertex-structure} follows.
\end{proof}

\begin{lemma}
\label{triangle}
Let $G$ be a connected, locally connected graph
and $C$~a cycle in~$G$.
Then for every vertex $w\in V(G)$ that has a neighbour outside~$C$
there is a pair of vertices $v,z$ (depending on~$w$)
such that $v\in N(w)\setminus V(C)$, $z\in V(C)$, and $wz, vz\in E(G)$.
\end{lemma}
\begin{proof}
Let $u$ be a vertex in $N(w)\setminus V(C)$. Since $G$ is locally connected, the ball $G_1(w)$ is 2-connected.
Then there is a path~$P$ in $G_1(w)-w$ with origin~$w^+$ and terminus~$u$.
Let $z$ be the vertex on $V(P)\cap V(C)$ such that
all other vertices on the path $z\overrightarrow Pu$ do not belong to~$C$,
and let $v$ denote the successor of~$z$ on~$\overrightarrow P$.
Then $v\in N(w)\setminus V(C)$, $z\in V(C)$ and $wz, vz\in E(G)$
\end{proof}

We also need the following result obtained in~\cite{asratian18a}.

\begin{proposition}[\cite{asratian18a}]
\label{proposition:ball}
If every ball of radius~$r$ in a graph~$G$ is $t$-connected, $t\geq 2$, then all balls of any radius~$r'>r$ in~$G$ are
$t$-connected, too.
\end{proposition}

Now we will show that for any $k\geq 3$ the set $\mathcal{G}(k)$ contains infinitely many Hamiltonian graphs.
First we describe in $\mathcal{G}(k)$ two infinite classes of Hamiltonian claw-free graphs where no ball of radius~2 is 2-connected.

\begin{example}
\label{claw-free1}
We define a graph $G(k,n)$, where $k\geq 3,n\geq 2$, as follows.
Its vertex set is $V_0\cup V_1\cup\dotsb\cup V_n$,
where $V_0,V_1,\dotsc,V_n$ are pairwise disjoint sets,
$V_0=\{v_1,\dots,v_n\}$ and $V_i=\{v_1^{i},\dots,v_k^i\}$, $i=1,\dots,n$.
The set~$V_i$ induces in $G(k,n)$ a complete graph, $i=1,\dots,n$.
For each $i=2,3,\dots,n$, the vertex~$v_i$ is adjacent to the first
$\left\lfloor \frac{k}{2}\right\rfloor$
vertices in~$V_i$ and the last $\left\lceil \frac{k}{2} \right\rceil$ vertices in~$V_{i-1}$.
Finally, $v_1$ is adjacent to the first $\left\lfloor \frac{k}{2}\right\rfloor$
vertices in~$V_{1}$ and the last $\left\lceil \frac{k}{2} \right\rceil$ vertices in~$V_{n}$.
\end{example}
\begin{example}
\label{claw-free2}
Now we define a graph $H(k,n)$, where $k\geq 3, n\geq 2$, as follows.
Its vertex set is $W_0\cup W_1\cup\dotsb\cup W_n$,
where $W_0,W_1,\dotsc,W_n$ are pairwise disjoint sets,
$W_0=\{u_1,\dots,u_{n},v_1,\dots,v_n\}$ and $|W_i|=k-1$, for $i=1,\dots,n$.
The set~$W_i$ induces in $H(k,n)$ a complete graph, $i=1,\dots,n$.
For each $i=1,\dots,n$, the vertices $u_i$ and~$v_{i}$ are adjacent to all
vertices in~$W_i$. Finally, $u_{n}v_1\in E(H(k,n))$ and $u_iv_{i+1}\in E(H(k,n))$, for $i=1,\dots,n-1$.
\end{example}

Clearly, all graphs in $\defset{G(k,n)}{n\geq 2}$ and $\defset{H(k,n)}{n\geq 2}$ are Hamiltonian and have $n(k+1)$ vertices,
connectivity~$2$ and diameter~$\floor[\big]{\tfrac{3n}2}$, $k=3,4,\dots$. Furthermore, they are not
locally connected, since the subgraph induced by the neighborhood $N(v_1)$ in each of these graphs is not connected.
Finally, the graph $G(3,n)$ is isomorphic to $H(3,n)$, for $n\geq 2$.
\medskip

Let us also note that for
every $k\geq 3$ the set $\mathcal{G}(k)$ contains Hamiltonian graphs
that are not claw-free.
First we mention graphs obtained from a complete bipartite graph $K_{k+1,k+1}$ by removing a perfect matching of $K_{k+1,k+1}$.
Next note that for every divisor $d\geq 3$ of the integer~$k$
a complete $r$-partite graph $K_{d,\dots,d}$, where $r=1+{k\over d}$,
is Hamiltonian and is not claw-free.
For example, the graphs $K_{3,3,3,3,3}, K_{4,4,4,4}, K_{6,6,6}$ and $K_{12,12}$
belong to the set $\mathcal{G}(12)$.
Also note that the required graphs can be obtained from a complete bipartite graph~$K_{k,k}$, $k\geq 4$, with
bipartition $(V_1,V_2)$ as follows: choose an integer $t<k/2$ and a matching $E(t)=\{x_1y_1,\dots,x_{2t}y_{2t}\}$ consisting of $2t$ edges, where
$x_1,\dots,x_{2t}\in V_1$, and $y_1,\dots,y_{2t}\in V_2$, and construct a graph
\[(G-E(t))+\{x_1x_2,y_1y_2,\dots,x_{2t-1}x_{2t},y_{2t-1}y_{2t}\}.\]
It is not difficult to verify that
such a graph is not claw-free and all its balls of radius~2 are 2-connected.
We have a stronger result in the case $k=4n$, $n\geq 1$ (see \cref{corollary:nonclaw,large-claw-free}).

\section{Hamiltonicity of graphs in \texorpdfstring{\boldmath$\mathcal{G}(k)$, $k\le5$\unboldmath}{𝓖(𝑘), 𝑘≤5}}

It is evident that all graphs in the set $\mathcal{G}(2)$ are Hamiltonian.

\begin{theorem}
\label{regular:three}
Every graph~$G$ in the set $\mathcal{G}(3)$ is Hamiltonian.
Moreover, either $G$~satisfies the conditions of \cref{asratian21}, or
$G$ is isomorphic to $H(3,n)$, for some $n\geq 2$.
\end{theorem}
\begin{proof}
Let $G$ be a graph in $\mathcal{G}(3)$.
If every ball of radius two in~$G$ is 2-connected
then $G$~satisfies the conditions of \cref{asratian21} and is thus Hamiltonian.

Suppose that not every ball of radius two in~$G$ is 2-connected,
and let $v_1$~be a cut vertex of a ball of radius~2 in~$G$.
Then by~\cref{lem:cut-vertex-structure},
$v_1$ is a cut vertex of the ball~$G_2(v_1)$,
the subgraph $G_2(v_1)-v_1$ has exactly two components with $3$~vertices each,
and each neighbor of~$v_1$ is adjacent to
all other vertices of the component of $G_2(v_1)-v_1$ it lies in.

Since $G$~is 3-regular,
one of the components of $G_2(v_1)-v_1$ will contain two neighbors of~$v_1$, denoted $w_{1}$ and~$w_{2}$,
and one other vertex which we call~$u_1$,
and the vertices $u_1$, $w_{1}$ and~$w_{2}$ will be adjacent to each other.
In~$G$, the vertex~$u_1$ will have one additional neighbor, which we call~$v_2$.
Clearly $d(v_1,v_2)>2$,
since $v_1$ is a cut vertex of $G_2(v_1)-v_1$.
But then $v_2$~is a cut vertex of the ball~$G_2(v_2)$,
and one of the components of $G_2(v_2)-v_2$ contains only one neighbor of~$v_2$, namely~$u_1$.
Let $F_1$ denote the subgraph induced by the set $\{v_1,u_1,v_2\}\cup W_1$, where $W_1=\{w_{1},w_{2}\}$.

Assume that we have already found in~$G$ $t\geq1$ isomorphic subgraphs $F_1,\dots,F_t$
and disjoint sets $W_0,W_1,\dots,W_t$ such that
\begin{itemize}
\item $W_0=\{u_1,\dots,u_t,v_1,\dots,v_t,v_{t+1}\}$ and $|W_1|=\dots=|W_t|$,
\item $V(F_i)=\{v_i,u_{i},v_{i+1}\}\cup W_i$, for $i=1,\dots,t$,
\item $v_i$ is a cut vertex of the ball $G_2(v_i)$, for $i=1,\dots,t+1$.
\end{itemize}

If $v_{t+1}\ne v_1$, then using for $v_{t+1}$ the same argument as for~$v_1$ we will find
a new subgraph~$F_{t+1}$
isomorphic to~$F_1$, where
$V(F_{t+1})=\{v_{t+1},u_{t+1},v_{t+2}\}\cup W_{t+1}$ and the vertex~$v_{t+2}$ is adjacent to~$u_{t+1}$,
$(W_{t+1}\cup\{u_{t+1}\})\cap \bigl(\bigcup_{i=1}^tV(F_i)\bigr)=\emptyset$ and
$v_{t+2}$ is a cut vertex of the ball $G_2(v_{t+2})$. Put $W_0\coloneqq W_0\cup \{u_{t+1},v_{t+2}\}$.

Since $G$~is finite, after finite number of steps we will find a vertex~$v_{n+1}$
such that $n\geq 2$ and $v_{n+1}=v_1$.
Then $\bigcup_{i=1}^nF_i=G$ and $G$ is isomorphic to~$H(3,n)$
and thus Hamiltonian.
The proof of the theorem is complete.
\end{proof}

Now we will show that all graphs in the set $\mathcal{G}(4)$ are Hamiltonian.
Moreover we will characterize all graphs in $\mathcal{G}(4)$ that do not satisfy
the conditions of \cref{asratian21}.

\begin{figure}
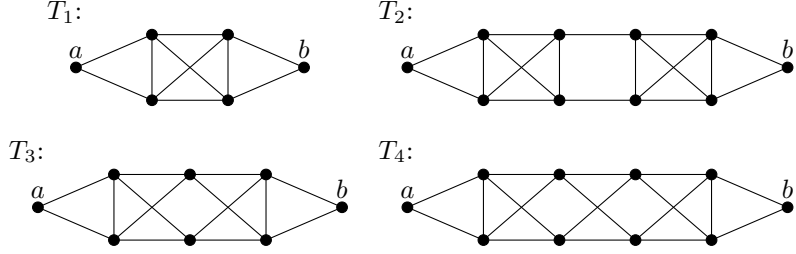

\centering
\figFparts
\caption{The graphs $T_1$, $T_2$, $T_3$, and $T_4$.}
\label{fig:4-reg-components}
\end{figure}

Let $F_1,\dots,F_n$, where $n\geq 2$, be a sequence of disjoint graphs where each~$F_i$ is isomorphic to one of the graphs in
\cref{fig:4-reg-components},
and let $u^1_i$ and $u^2_{i}$ be the vertices in~$F_i$ corresponding to the vertices $a$ and~$b$ in this isomorphism.
Consider a graph obtained from
$F_1,\dots,F_n$ by identifying the vertices $u^2_i$ and $u^1_{i+1}$ and denoting the obtained new vertex by~$u_{i+1}$, $i=1,\dots,n$, where
$u^1_{n+1}=u_{n+1}=u_1$. The set of all graphs obtaining in such a way we denote by $\mathcal{F}(4,n)$.
Clearly, $G(4,n)\in \mathcal{F}(4,n)$
and the diameter of any $G\in\mathcal{F}(4,n)$ satisfies
$\floor{\tfrac{3n}2}\le\diam(G)\le\floor{\tfrac{5n}2}$,
for each $n\geq 2$.
All graphs from the set $\mathcal{F}(4,2)$ can be seen in \cref{fig:4-reg-with-2-segments}.

\begin{figure}
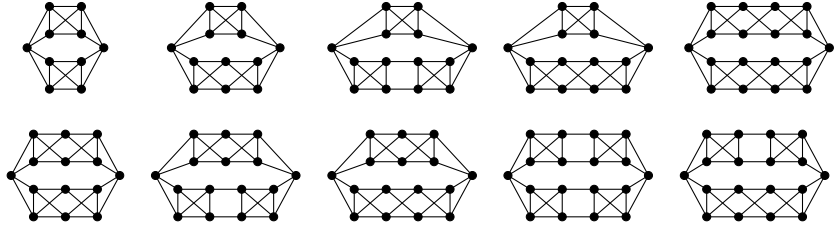

\centering
\figFfourtwo
\caption{The set $\mathcal{F}(4,2)$}
\label{fig:4-reg-with-2-segments}
\end{figure}

\begin{theorem}
\label{four:regular}
Every graph~$G$ in the set $\mathcal{G}(4)$ is Hamiltonian.
Moreover, either $G$~satisfies the conditions of \cref{asratian21}, or
$G$ is isomorphic to $H(4,n)$ or one of graphs in $\mathcal{F}(4,n)$, for some $n\geq 2$.
\end{theorem}

\begin{proof}
Let $G$ be a graph in $\mathcal{G}(4)$.
If every ball of radius two in~$G$ is 2-connected
then $G$~satisfies the conditions of \cref{asratian21} and is thus Hamiltonian.

Suppose that not every ball of radius two in~$G$ is 2-connected,
and let $v_1$~be a cut vertex of a ball of radius~2 in~$G$.
Then by~\cref{lem:cut-vertex-structure},
$v_1$ is a cut vertex of the ball~$G_2(v_1)$,
the subgraph $G_2(v_1)-v_1$ has exactly two components with $4$~vertices each,
and each neighbor of~$v_1$ is adjacent to
all other vertices of the component of $G_2(v_1)-v_1$ it lies in.

Since $G$~is 4-regular,
$v_1$ will have either one neighbor in one of the components of $G_2(v_1)-v_1$ and three in the other,
or two neighbors in each component.

\begin{case}
\label{case:113}
$v_1$ has one neighbor in one of the components of $G_2(v_1)-v_1$ and three in the other.
\end{case}
Consider the component of $G_2(v_1)-v_1$ that contains three neighbors of~$v_1$, denoted $w_{1}, w_{2}, w_{3}$,
and one other vertex which we will call~$u_1$.
By \cref{lem:cut-vertex-structure}, the vertices $u_1, w_{1}, w_{2}, w_{3}$ will be adjacent to each other.
In~$G$, the vertex~$u_1$ will have one additional neighbor, which we will call~$v_2$.
Clearly $d(v_1,v_2)>2$,
since $v_1$~is a cut vertex of $G_2(v_1)$.
But then $v_2$~is a cut vertex of the ball~$G_2(v_2)$,
and one of the components of $G_2(v_2)-v_2$ contains only one neighbor of~$v_2$, namely~$u_1$.
Denote by $F_1$ the subgraph of~$G$ induced by the set $\{v_1,u_1,v_2\}\cup W_1$
where $W_1=\{w_{1}, w_{2}, w_{3}\}$.

By repeating the argument that we used in the proof of \cref{regular:three} we can show that
the graph~$G$ is a union of $n\ge 2$
subgraphs $F_1,\dots,F_n$ such that $F_2,\dots,F_n$ are isomorphic to~$F_1$,
$|V(F_i)\cap V(F_{i+1})|=1$ for $i=1,\dots,n$ (with $F_{n+1}=F_1$)
and $|V(F_i)\cap V(F_{j})|=\emptyset$, if $|i-j|\not=1\mod n$.

It is easy to see that the resulting graph is isomorphic to~$H(4,n)$.

\begin{case}
\label{case:4-reg-22}
$v_1$ has two neighbors in each component of $G_2(v_1)-v_1$.
\end{case}
Let $w_1$ and~$w_2$ be the neighbors of~$v_1$ in one of the components of $G_2(v_1)-v_1$,
and let $u_1$ and~$u_2$ be the other two vertices in that component.
As noted above, $w_1$ and~$w_2$ are adjacent to each other and to $u_1$ and~$u_2$.
Let $S=\set{w_1,w_2,u_1,u_2}$.

\begin{subcase}
$u_1u_2\in E(G)$.
\end{subcase}
In this case $u_1$ and~$u_2$ each have one neighbor outside~$S$.

\begin{subsubcase}
$u_1$ and $u_2$ are adjacent to the same vertex~$v_2\notin S$.
\end{subsubcase}
In this case the graph~$F_1$ induced by the set $S\cup\set{v_1,v_2}$ is isomorphic to~$T_1$.
Clearly $d(v_1,v_2)>2$, since $v_1$~is a cut vertex of $G_2(v_1)$.

\begin{subsubcase}
$u_1$ and $u_2$ are adjacent to different vertices outside~$S$.
\end{subsubcase}
Let $x_1$ be the neighbor of~$u_1$ outside~$S$,
and let $x_2$ be the neighbor of~$u_2$ outside~$S$.
Clearly $x_1x_2\in E(G)$,
since otherwise there are five vertices at distance two from~$u_1$, namely $x_2$,~$v_1$ and three neighbors of~$x_1$ outside~$S$, contradicting the conditions of the theorem.
Let $y_1$ and~$y_2$ be the two neighbors of~$x_1$ outside~$S\cup\set{x_2}$.
Then $w_1$, $w_2$, $u_2$
and the vertices in the set~$N(y_1)\setminus\set{x_2,y_2}$ are at distance~2 from~$x_1$.
This means that the latter set can contain at most one vertex,
so $y_1$~is adjacent to $x_2$ and~$y_2$.
Similarly we see that $y_2$~is also adjacent to~$x_2$.
Now $y_1$ and~$y_2$ have one neighbor each outside $\set{x_1,x_2,y_1,y_2}$.
Let $v_2$ and~$z$ be the neighbors of~$y_1$ and~$y_2$, respectively, outside this set.
If~$v_2\not=z$ then there are five vertices, namely $w_1$, $w_2$, $u_2$, $v_2$ and~$z$, at distance~2 from~$x_1$,
a contradiction.
Thus $v_2=z$,
which means that the subgraph~$F_1$ induced by $S\cup\set{x_1,x_2,y_1,y_2,v_1,v_2}$
is isomorphic to~$T_2$.
Note that $v_2\ne v_1$,
since otherwise $G=G_2(v_1)$, contradicting the fact that $v_1$~is a cut vertex of~$G_2(v_1)$.
This means that $u_1$ and~$u_2$ are at distance~3 from~$v_2$,
so $v_2$~is a cut vertex of~$G_2(v_2)$.

\begin{subcase}
$u_1u_2\notin E(G)$.
\end{subcase}
In this case $u_1$ and~$u_2$ each has two neighbors outside~$S$,
and these neighbors are at distance~2 from~$w_1$.
Clearly $N(u_1)=N(u_2)$, since otherwise there are more than four vertices at distance~2 from~$w_1$.
Let $x_1$ and~$x_2$ be the neighbors of $u_1$ and~$u_2$ outside~$S$.

\begin{subsubcase}
$x_1x_2\in E(G)$.
\end{subsubcase}
In this case $x_1$ and~$x_2$ both have one more neighbor.
Let~$v_2$ be the neighbor of~$x_1$ outside~$S\cup\set{x_2}$.
If $x_2v_2\notin E(G)$, then there are five vertices at distance~2 from~$x_1$, namely
$w_1$, $w_2$, and the three other neighbors of~$v_2$.
Thus $x_2v_2\in E(G)$,
which means that the subgraph induced by $S\cup\set{x_1,x_2,v_1,v_2}$
is isomorphic to~$T_3$.
Note that $v_2\notin M_1(v_1)$,
since otherwise $G=G_2(v_1)$, contradicting the fact that $v_1$~is a cut vertex of~$G_2(v_1)$.
This means that $w_1$ and~$w_2$ are at distance~3 from~$v_2$,
so $v_2$~is a cut vertex of~$G_2(v_2)$.

\begin{subsubcase}
$x_1x_2\notin E(G)$.
\end{subsubcase}
In this case $x_1$ and~$x_2$ each has two neighbors outside~$S$,
and these neighbors are at distance~2 from~$u_1$.
Clearly $N(x_1)=N(x_2)$, since otherwise there are more than four vertices at distance~2 from~$u_1$.
Let $y_1$ and~$y_2$ be the neighbors of $x_1$ and~$x_2$ outside~$S$.
Then the vertices $w_1$, $w_2$, $x_2$,
and any neighbors of $y_1$ and~$y_2$ that are not adjacent to~$x_1$ are at distance~2 from~$x_1$.
Thus $y_1$ and~$y_2$ are adjacent to each other and to one common additional vertex~$v_2$,
which means that the subgraph induced by $S\cup\set{x_1,x_2,y_1,y_2,v_1,v_2}$
is isomorphic to~$T_4$.
Note that $v_2\ne v_1$,
since otherwise $G=G_2(v_1)$, contradicting the fact that $v_1$~is a cut vertex of~$G_2(v_1)$.
This means that $u_1$ and~$u_2$ are at distance~3 from~$v_2$,
so $v_2$~is a cut vertex of~$G_2(v_2)$.

\medskip
To conclude \cref{case:4-reg-22},
we see that in all subcases we found a subgraph~$F_1$ that is isomorphic to one of the graphs
$T_1, T_2, T_3,T_4$ and has only two vertices of degree~2, namely $v_1$ and~$v_2$.

Assume that we have already
found $t\geq 1$ subgraphs $F_1,\dots,F_t$ each of which is
isomorphic to one of the graphs $T_1, T_2, T_3,T_4$ such that
\begin{itemize}
\item $F_i$ has only two vertices of degree~2 in~$F_i$, namely $v_i$ and~$v_{i+1}$, $i=1,\dots,t$,
\item $v_i$ is a cut vertex of the ball $G_2(v_i)$, for $i=1,\dots,t+1$,
\item if $t\ge 2$ then $V(F_{i})\cap V(F_{i+1})=\{v_{i+1}\}$, for $i=1,\dots,t-1$.
\end{itemize}

If $v_{t+1}\ne v_1$, then using for $v_{t+1}$ the same argument as for~$v_1$ we will find in~$G$ a subgraph~$F_{t+1}$
isomorphic to one of the graphs $T_1,T_2,T_3,T_4$ such that $V(F_{t+1})\cap V(F_t)=\{v_{t+1}\}$
and $v_{t+1}$ is a cut vertex of the ball $G_2(v_{t+1})$.

Since $G$~is finite, after finite number of steps we will find subgraphs $F_1,\dots,F_n$ and vertices
$v_1,\dots,v_{n+1}$
such that $v_{n+1}=v_1$ and $\bigcup_{i=1}^nF_i=G$.
This means that $G\in\mathcal{F}(4,n)$
and therefore $G$~is Hamiltonian.
The proof of the theorem is complete.
\end{proof}

Note that \cref{regular:three} and \cref{four:regular} imply \cref{pippert74}. This holds because
every connected, locally connected, $k$-regular graph with $k\in\{3,4\}$ is
the square of a cycle~$C_n$ of length~$n$, $n\geq 4$, (see~\cite{gordon11}),
and the square of~$C_n$ belongs to $\mathcal{G}(3)$
if $n=4$ and to $\mathcal{G}(4)$ if $n>4$.

\begin{corollary}
\label{corollary:nonclaw}
The set $\mathcal{G}(4)$ contains an infinite set of Hamiltonian graphs that are not claw-free or
locally connected.
\end{corollary}
\begin{proof}
Consider the set of all graphs in $\defset{\mathcal{F}(4,n)}{n\geq 2}$ containing an induced subgraph~$T_3$ or~$T_4$.
Clearly this set is infinite and no graph in this set is claw-free or locally connected.
\end{proof}

\begin{theorem}
\label{five:regular}
Every graph~$G$ in the set $\mathcal{G}(5)$ is Hamiltonian.
Moreover, either $G$~satisfies the conditions of \cref{asratian21}, or
$G$~is isomorphic to one of the graphs
$G(5,n)$ and $H(5,n)$, for some $n\geq 2$.
\end{theorem}

\begin{proof}
Let $G$ be a graph in $\mathcal{G}(5)$.
If every ball of radius two in~$G$ is 2-connected
then $G$~satisfies the conditions of \cref{asratian21} and is thus Hamiltonian.

Suppose that not every ball of radius two in~$G$ is 2-connected,
and let $v_1$~be a cut vertex of a ball of radius~2 in~$G$.
Then by~\cref{lem:cut-vertex-structure},
$v_1$ is a cut vertex of the ball~$G_2(v_1)$,
the subgraph $G_2(v_1)-v_1$ has exactly two components with $5$~vertices each,
and each neighbor of~$v_1$ is adjacent to
all other vertices of the component of $G_2(v_1)-v_1$ it lies in.
Since $G$~is 5-regular,
$v_1$ will have either one neighbor in one of the components of $G_2(v_1)-v_1$ and four in the other,
or two neighbors in one component and three in the other.

\begin{case}
\label{case:114}
$v_1$ has one neighbor in one of the components of $G_2(v_1)-v_1$ and four in the other.
\end{case}
Consider the component of $G_2(v_1)-v_1$ which contains four neighbors of~$v_1$.
Let $W_1=\{w_{1},w_{2},w_{3},w_{4}\}$ denote the set of four neighbors of~$v_1$ in this component, and let $u_1$ denote the fifth vertex in the same component. By~\cref{lem:cut-vertex-structure},
all of these five vertices are adjacent to each other.
In~$G$, the vertex~$u_1$ will have one additional neighbor, which we will call~$v_2$.
Clearly $d(v_1,v_2)>2$, since
$v_1$~is a cut vertex of $G_2(v_1)-v_1$.
But then $v_2$~is a cut vertex of the ball~$G_2(v_2)$,
and one of the components of $G_2(v_2)-v_2$ contains only one neighbor of~$v_2$, namely~$u_1$.
Denote by $F_1$ the subgraph of~$G$ induced by the set $\{v_1,u_1,v_2\}\cup W_1$.

By repeating the argument that we used in the proof of \cref{regular:three} we can show that
in~$G$ there are $n\geq 2$
subgraphs $F_1,\dots,F_n$ and $n$~vertices $v_1,\dots,v_n$ such that
$\bigcup_{i=1}^nF_i=G$, each $F_i$ is isomorphic to~$F_1$, $2\leq i\leq n$,
$V(F_i)\cap V(F_{i+1})=\{v_{i+1}\}$, for $i=1,\dots,n$,
and $V(F_i)\cap V(F_{j})=\emptyset$, for all $|i-j|\ne 1 \mod n$, where we consider $v_{n+1}=v_1$.

It is easy to see that the resulting graph is isomorphic to~$H(5,n)$ and therefore is Hamiltonian.

\begin{case}
$v_1$ has two neighbors in one of the components of $G_2(v_1)-v_1$ and three in the other.
\end{case}
Consider the component of $G_2(v_1)-v_1$ which has two neighbors of~$v_1$.
Let $w_1$ and~$w_2$ be the two neighbors of~$v_1$ in this component,
and let the other three vertices in the component be~$z_1$, $z_2$, and~$z_3$.
As noted above, $w_1$ and~$w_2$ are adjacent to each other and to $z_1$, $z_2$, and~$z_3$.
At distance~2 from~$w_1$ are the three other neighbors of~$v_1$, and all neighbors of $z_1$, $z_2$, and~$z_3$.
outside the set $S=\set{w_1,w_2,z_1,z_2,z_3}$.
Thus $z_1$, $z_2$, and~$z_3$ can together have
at most two neighbors outside~$S$.
Since $G$~is 5-regular,
this means that each of the vertices $z_1$, $z_2$, and~$z_3$
must be adjacent to at least one of the other two,
which means that one of them, say~$z_1$, must be adjacent to both $z_2$ and~$z_3$.

Suppose that $z_2z_3\notin E(G)$.
Then $z_2$ and~$z_3$ are both adjacent to two vertices $x_1$ and~$x_2$ outside~$S$,
and~$z_1$ will be adjacent to one of them, say~$x_1$.
Since $G$ is 5-regular,
$x_2$ will have at least one neighbor~$y$ outside $S\cup\set{x_1}$ that is not adjacent to~$x_1$.
If $x_1x_2\in E(G)$, then $y$~will have at least three neighbors outside $M_1(x_2)$,
and these vertices together with $w_1$, $w_2$, and~$z_1$ are all at distance~2 from~$x_2$,
contradicting the fact that $G\in \mathcal{G}(5)$.
If instead $x_1x_2\notin E(G)$, then $y$~will have at least two neighbors outside $M_1(x_2)$,
and these vertices together with $w_1$, $w_2$, $z_1$ and~$x_1$ are all at distance~2 from~$x_2$,
contradicting the fact that $G\in \mathcal{G}(5)$.

Thus we can conclude that $z_2z_3\in E(G)$.
Since $z_1$, $z_2$, and~$u_z$ are all adjacent to each other,
each of them has exactly one neighbor outside~$S$,
and as mentioned before they can together have at most two neighbors outside~$S$.

Suppose that $z_1$, $z_2$, and~$u_z$ are not all adjacent to the same vertex outside~$S$.
Then w.l.o.g. we can assume that $z_1$ is adjacent to a vertex~$x_1\notin S$
and $z_2$ and~$z_3$ are adjacent to another vertex~$x_2\notin S$.
We can see that $x_1x_2\in E(G)$,
since otherwise $z_1$ would have six vertices at distance~2:
$v_1$, $x_2$, and the four remaining neighbors of~$x_1$.
Thus $x_1$ has three neighbors apart from $z_1$ and~$x_2$,
we will call them $y_1$, $y_2$, and~$y_3$.
Since $G$~is 5-regular, $y_1$ must have a neighbor~$z_4$ outside $\set{x_1,x_2,y_1,y_2,y_3}$.
Then $x_1$ has five vertices at distance~2,
namely $w_1$, $w_2$, $z_2$, $z_3$, and~$z_4$.
Thus all remaining neighbors of $y_1$, $y_2$, and~$y_3$ must be in $\set{x_2,y_1,y_2,y_3,z_4}$,
which means that $y_1$, $y_2$, and~$y_3$ are all adjacent to~$x_2$.
But then $d(x_2)=6$, a contradiction.

Thus, we can conclude that $z_1$, $z_2$, and~$z_3$ are all adjacent to the same vertex outside~$S$,
denoted~$v_2$. Set $V_1=\{w_1,w_2,z_1,z_2,z_3\}$.
Clearly, the subgraph induced by~$V_1$ is a complete graph.

Assume that we have already found in~$G$ a sequence of disjoint sets $V_0,V_1,\dots,V_t$ such that
$V_0=\{v_1,\dots,v_{t+1}\}$,
and for $i=1,\dots,t$, the subgraph induced by~$V_i$ is complete, $v_i$ is adjacent to two vertices in~$V_i$, and the other vertices in $V_i$ are adjacent to~$v_{i+1}$.

If $v_{t+1}\not=v_1$ then using the same argument as we used for~$v_1$
we will find a set~$V_{t+1}$ and a vertex $v_{t+2}\notin V_{t+1}$ such that the subgraph induced by~$V_{t+1}$ is complete,
$V_{t+1}\cap \bigl(\bigcup_{i=0}^tV_i\bigr)=\emptyset$ and $v_{t+2}$ is adjacent to those vertices in~$V_{t+1}$ that are not
neighbors of~$v_{t+1}$. Then put $V_0\coloneqq V_0\cup \{v_{t+2}\}$.

Since $G$~is finite, after finite number of steps we will find $n+1$ disjoint sets $V_0,V_1,\dots,V_n$ and
a vertex~$v_{n+1}$ such that $V_0=\{v_1,\dots,v_n\}$, $v_{n+1}=v_1$ and $V_0\cup V_1\cup\dots\cup V_n=V(G)$.
Thus $G$ is isomorphic to~$G(5,n)$
and therefore Hamiltonian.
The proof of the theorem is complete.
\end{proof}

\begin{remark}
The graphs $G(5,n)$ and $H(5,n)$ are not locally connected for $n\geq 2$ and therefore do not satisfy the conditions of Kikust's theorem (\cref{kikust75}).
On the other hand there exist locally connected, 5-regular graphs which do not satisfy \cref{five:regular}.
Therefore \cref{five:regular} and Kikust's theorem are incomparable to each other in the sense that neither theorem implies the other.
\end{remark}

\section{Hamiltonicity of graphs in \texorpdfstring{\boldmath$\mathcal{G}(k)$, $k\ge6$\unboldmath}{𝓖(𝑘), 𝑘≥6}}

The next result shows that
the condition of 2-connectedness of balls of radius~2 in {\cref{asratian21}}
cannot be omitted in the case $ k\geq 6$.

\begin{proposition}
\label{nonHamiltonian}
For any $k\geq 6$, the set $\mathcal{G}(k)$ contains non-Hamiltonian graphs.
\end{proposition}
\begin{proof}
Pick three integers $k_1, k_2, k_3\ge {2}$ such that $k_1+k_2+k_3=k$.
We define three disjoint graphs $H_1, H_2, H_3$ as follows:
For each~$i=1,2,3$,
the vertex set of~$H_i$ is
$U_i\cup V_i\cup V'_i\cup U'_i\cup\set{w_i,w'_i}$ where

\begin{figure}
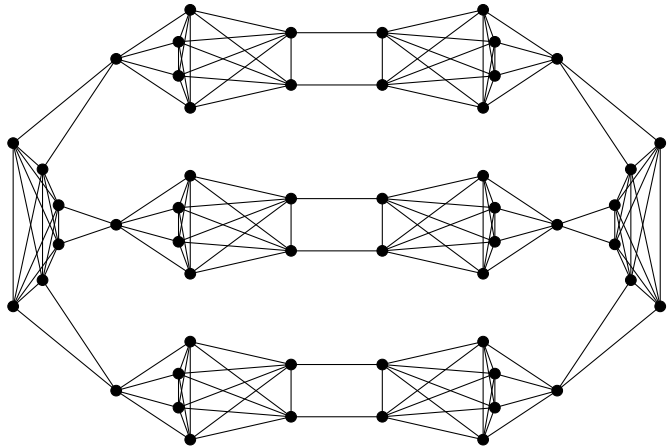

\centering
\figcounterexample
\caption{A non-Hamiltonian graph in~$\mathcal{G}(6)$.}
\label{fig:counterexample6}
\end{figure}

\begin{itemize}
\item the sets $U_i$, $V_i$, $V'_i$, $U'_i$,~and $\set{w_i,w'_i}$
are pairwise disjoint,
\item $\card{V_i}=\card{V'_i}=k_i$ and $\card{U_i}=\card{U'_i}=k-k_i$,
\item each of the sets~$U_i$, $V_i$, $V'_i$ and~$U'_i$
induces a complete subgraph,
\item every vertex in~$U_i$ is adjacent to $w_i$ and to all vertices in~$V_i$,
\item every vertex in~$U'_i$ is adjacent to $w'_i$ and to all vertices in~$V'_i$.
\item every vertex in~$V_i$ is adjacent to exactly one vertex in~$V'_i$, and vice versa.
\end{itemize}
We also let $\tilde H$ and $\tilde H'$ be two copies of~$K_k$,
disjoint from each other and from $H_1$, $H_2$, and~$H_3$.

Now we will use the parts defined above to construct
a non-Hamiltonian graph in~$\mathcal{G}(k)$.
We will start with the graph $\tilde H\cup H_1\cup H_2\cup H_3\cup \tilde H'$.
To this graph we will add edges such that for each $i=1,2,3$
the vertex~$w_i$ is adjacent to $k_i$~vertices of~$\tilde H$
and each vertex of $\tilde H$ is adjacent to exactly one of $w_1$, $w_2$, and~$w_3$,
and similarly for~$\tilde H'$ and $w'_1$, $w'_2$, and~$w'_3$.
The resulting graph lies in~$\mathcal{G}(k)$ but is not Hamiltonian.
The unique graph for $k=6$ (with $k_1=k_2=k_3=2$) can be seen in \cref{fig:counterexample6}.
\end{proof}

For $k=6$ we furthermore have the following:

\begin{theorem}
\label{NP-complete}
The problem of determining whether there exists a Hamilton cycle in
a graph from $\mathcal{G}(6)$ is NP-complete.
\end{theorem}
\begin{proof}
Akiyama, Nishizeki, and Saito~\cite{akiyama80} proved that
the problem of determinining whether there exists a Hamilton cycle in
a 2-connected, 3-regular, bipartite, planar graph is NP-complete.
Thus, to prove the theorem we only need to provide a polynomial-time reduction
of this problem to our problem.

Let $G$~be any 2-connected,
3-regular, bipartite, planar graph.
We will construct a graph~$G'\in\mathcal{G}(6)$ that is Hamiltonian if and only if $G$~is.
Let $(X,Y)$ be a bipartition of~$G$.
We set
\[V(G')=\bigcup_{x\in X}U_x \cup \bigcup_{y\in Y}V_y \cup \defset{w_e}{e\in E(G)},\]
where each $U_x$ is a set of four vertices and each $V_y$ is a set of six vertices
such that $U_x\cap V_y=\emptyset=U_x\cap U_z=V_y\cap V_s$ for all $x,z\in X$ and $y,s\in Y$.
The edges of~$G'$ are defined as follows:
\begin{itemize}
\item each of the sets $U_x$ and~$V_y$ induces a complete graph in~$G'$, for $x\in X$, $y\in Y$\!, and
\item if $e=xy$ is an edge in~$G$ with $x\in X$ and $y\in Y$, then the vertex~$w_e$ in~$G'$
is adjacent to all vertices in~$U_x$ and exactly two vertices in~$V_y$ with the additional requirement that
if $e'$ and~$e{''}$ are two other edges in~$G$ incident to~$y$, then every vertex in~$V_y$ is adjacent to
exactly one of the vertices $w_e, w_{e'}$, and~$w_{e{''}}$, see \cref{fig:reduction}.
\end{itemize}

\begin{figure}
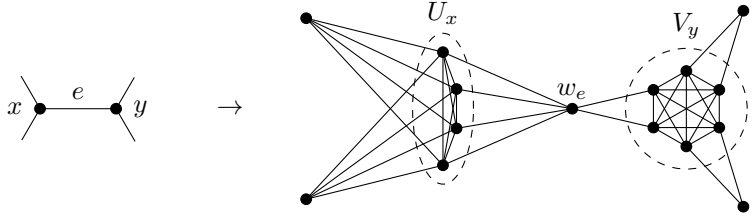

\centering
\figreduction
\caption{The transformation from $G$ to~$G'$.}
\label{fig:reduction}
\end{figure}

It is easy to verify that $G'\in\mathcal{G}(6)$.
% |N₂(u)|=6 for all u
It is also clear that $G'$ can be constructed from~$G$ in polynomial time.

Before we prove that $G'$ is Hamiltonian if and only if $G$~is,
we will introduce some practical notation.
Let $e=xy\in E(G)$.
In any Hamilton cycle of~$G'$, the vertex~$w_e$ is incident to exactly two edges of the cycle.
If one of these edges joins~$w_e$ to a vertex of~$U_x$
and the other edge joins~$w_e$ to a vertex of~$V_y$,
we say that $w_e$~is \emph{crossed}; otherwise it is not crossed.
If $e_1,e_2,e_3$ are three edges incident to a vertex in~$G$,
then $\set{w_{e_1},w_{e_2},w_{e_3}}$ is a cut set of~$G'$.
This means that in any Hamilton cycle of~$G'$,
exactly two of $w_{e_1},w_{e_2},w_{e_3}$ are crossed.

Suppose there is a Hamilton cycle~$C'$ in~$G'$.
Then we will construct in~$G$ a subgraph~$C$ by including an edge~$e\in E(G)$
in~$C$
if $w_e$ is crossed by~$C'$.
By the previous discussion,
$C$~will be a 2-regular spanning subgraph of~$G$.
Also, for any two vertices~$z_1$ and~$z_2$ in~$G$
we can find a path with edges~$e_1,\dotsc,e_p$
such that $w_{e_1},\dotsc,w_{e_p}$~are crossed by~$C'$,
since $C'$~is a Hamilton cycle in~$G'$.
Thus $C$ is connected,
so it will form a Hamilton cycle of~$G$.

Suppose now that $G$~is Hamiltonian,
and let
$\overrightarrow C$ denote a Hamilton cycle in~$G$ with a given orientation.
Consider an edge $e=xy$ in $\overrightarrow C$
and assume that the edge preceding~$x$ in the cycle is~$d$,
that the edge succeding~$y$ is~$f$,
and that the other two edges incident to~$x$ and~$y$ are $a$ and~$b$, respectively.
It is not hard to find a path from $w_d$ to~$w_f$ in~$G'$
that covers all vertices of $U_x\cup V_y\cup\set{w_d,w_e,w_f,w_b}$,
see \cref{fig:hamreduction}.
By representing the Hamilton cycle $\overrightarrow C$ as
$\overrightarrow C=x_1y_1x_2y_2\cdots x_ny_nx_1$
and using this method for every pair $x_ky_k$,
we can construct a Hamilton cycle of~$G'$.

\begin{figure}
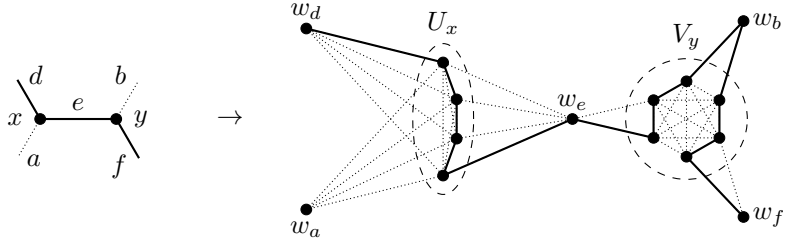

\centering
\fighamreduction
\caption{The transformation from a Hamilton cycle in~$G$ to one in~$G'$.}
\label{fig:hamreduction}
\end{figure}

We can conclude that this construction gives us a polynomial time reduction of
the Hamilton cycle problem for 2-connected cubic bipartite planar graphs
to the Hamilton cycle problem for graphs in $\mathcal{G}(6)$,
and thus that the latter problem is NP-complete.
\end{proof}

By \cref{pippert74,irzhav}, all connected, locally connected, $k$-regular graphs are Hamiltonian if $k\leq 5$.
We now formulate a result describing a new class of Hamiltonian, locally connected, $k$-regular graphs
with $k\geq 6$.

\begin{theorem}
\label{like:bondy}
Every locally connected graph~$G$ in $\mathcal{G}(k)$, $k\geq 6$, is Hamiltonian.
Moreover, for every non-Hamiltonian cycle~$C$ in~$G$ there exists a cycle of length
$|V(C)|+\ell$ in~$G$, $\ell\in\set{1,2}$, containing the vertices of~$C$.
\end{theorem}

\begin{proof}
Since $G$ is locally connected, every vertex of~$G$ lies on a triangle. Suppose that for some $n<|V(G)|$, $G$~has a cycle~$C$ of length~$n$ but
\begin{equation}
\label{E3}
\text{$G$ has no cycle of length $n+1$ or $n+2$ containing the vertices of~$C$.}
\end{equation}
Consider a vertex~$w_1$ on~$C$ with $N(w_1)\setminus V(C)\not=\emptyset$.
Such a vertex exists because $G$~is connected.
By \cref{triangle}, there exists a pair of vertices $v,z$
such that $v\in N(w_1)\setminus V(C)$, $z\in V(C)$ and $w_1z, vz\in E(G)$.

Set $W=N(v)\cap V(C)=\{w_1,\dots,w_p\}$, $p\geq 2$, where the vertices $w_1,\dots,w_p$ occur on
$\overrightarrow C$ in the order of their indices.
(We consider $w_{p+1}=w_1$.)

\begin{claim}
\label{snitt2}
$|N(v)\cap N(w_1^+)|\geq 2$.
\end{claim}

\begin{proof}
Suppose that $N(v)\cap N(w_1^+)=\{w_1\}.$
By \cref{E3}, $vw_1^+\notin E(G)$ and, by \cref{lemma1},
$|N(v)\cap N(w_1^+)|\geq |M_2(w_1)\setminus (N(v)\cup N(w_1^+))|-1.$
Then $N(v)\cap N(w_1^+)=\{w_1\}$ implies
$M_2(w_1)\setminus (N(v)\cup N(w_1^+))=\{v,w_1^+\}$.
This means that $w_1^+$ is adjacent to every vertex in $M_2(w_1)\setminus (M_1(v)\cup \{w_1^+\})$.
By \cref{E3}, the vertex~$v$ is not adjacent to~$z^+$.
Then $z^+$ belongs to the set $M_2(w_1)\setminus (M_1(v)\cup \{w_1^+\})$ and, therefore, $w_1^+z^+\in E(G)$.
But then $G$ has the cycle $w_1vz\overleftarrow Cw_1^+z^+\overrightarrow Cw_1$ of length $n+1$
containing $V(C)$, a contradiction.
\end{proof}

\begin{claim}
$w_i^+=w^-_{i+1}$ for each $i=1,\dotsc,p$, that is, $n=2p$ and $v$~is adjacent to every second vertex of~$C$.
\end{claim}
\begin{proof}
Set $W^+=\{w_1^+,\dotsc,w_p^+\}$.
We will count the number of edges between $W^+$ and~$W$ which we denote by $e(W^+,W)$.
By the assumption \cref{E3}, the set $W^+\cup\{v\}$ is
independent and $N(w_i^+)\cap N(v)\cap (V(G)\setminus V(C))=\emptyset$, for $1\leq i\leq p$.
Moreover, for each
$i=1,\dots,p$, we have
$d(v,w^+_i)=2$ and $w_i\in N(v)\cap N(w_i^+)$, so by the hypothesis of this theorem and by \cref{lemma1},
\begin{equation}
|N(v)\cap N(w_i^+)|\geq |M_2(w_i)\setminus (N(v)\cup N(w_i^+)|-1.
\label{E2}
\end{equation}
Obviously,
\begin{equation}
N(w_i)\cap W^+\subseteq M_2(w_i)\setminus (N(v)\cup N(w_i^+)\cup \{v\}).
\end{equation}
Thus,
$|N(w_i)\cap W^+|\leq |M_2(w_i)\setminus (N(v)\cup N(w_i^+))|-1.$
This and \cref{E2} imply that
$|N(w_i)\cap W^+|\leq |N(v)\cap N(w_i^+)|.$
Hence,
\begin{equation}
e(W^+, W)=\sum_{i=1}^k|N(w_i)\cap W^+|\leq \sum_{i=1}^k|N(v)\cap N(w_i^+)|=e(W^+,W).
\end{equation}
It follows, for each $i=1,\dots,p$, that
\begin{equation}
M_2(w_i)\setminus (N(v)\cup N(w_i^+)\cup \{v\})=N(w_i)\cap W^+\subseteq W^+.
\label{E8}
\end{equation}

Noting that $p\geq2$ and, by \cref{snitt2},
$|N(w_1^+)\cap N(v)|\geq2$, we now prove by contradiction that
$w_i^+=w^-_{i+1}$ for each $i=1,\dotsc,p$.

Assume without loss of generality that $w_1^+\neq w_2^-$, whence
$w_2^-\notin W^+$. Observe that $w_2^-\in N(w_2^+)$, because otherwise, by \cref{E8}, $w_2^-\in W^+$.
Clearly, the assumption~\cref{E3} implies $w^-_2w^-_3 \notin E(G)$.
Hence $w_2^+\neq w_3^-$.
Repetition of this argument shows that $w_i^+\neq w^-_{i+1}$ and
$w_i^+w_i^-\in E(G)$ for all $i\in\{\,1,\ldots,p\,\}$.
By~\cref{snitt2}, $N(w_1^+)\cap N(v)$ contains a vertex $x\neq w_1$.
Clearly, $x\in V(C)$, because otherwise $G$ contains a cycle of length $n+2$ containing the vertices of~$C$.
Then $x=w_i$ for some $i\geq 2$ and $G$~contains
the cycle $w_1vw_iw_1^+\overrightarrow Cw_i^-w_i^+\overrightarrow Cw_1$ of length~$n+1$
containing $V(C)$.
This contradiction proves~that $w_i^+=w^-_{i+1}$ for each $i=1,\dotsc,p$, and $n=2p$.
\end{proof}

We continue to prove the theorem. First we will show that the vertex~$w_i^+$ has no neighbor outside
the cycle~$C$, for each $i=1,\dotsc,p$.
Suppose that the condition $N(w_t^+)\setminus V(C)\ne\emptyset$ holds for some~$t$.
Since the vertex~$w_1$ was picked arbitrarily in $V(C)$, we can conclude, using similar arguments, that
$w_t^+$ has a neighbor $u\in N(w_t^+)\setminus V(C)$ that is adjacent to every second vertex of~$C$ as well. More precisely, $N(u)\cap V(C)=W^+$.
But then $G$ has the cycle $w_1vw_2w_1^+uw_2^+\overrightarrow Cw_1$ of length $n+2$
containing $V(C)$.
This contradiction implies that $N(w_i^+)\subseteq W$ and $d(w_i^+)\leq p$,
for each $i=1,\ldots,p$.
On the other hand, $W\subseteq N(v)$, that is, $d(v)\geq p$.
These condition and $k$-regularity of~$G$ imply that $N(v)=W$,
$d(v)=p=k$ and $N(w_i^+)=W$, for each $i=1,\dotsc,p$.
Then $w_1$ is adjacent to each $w_i^+$, $i=1,\dotsc,p$,
as well as to~$v$.
But this
implies that $d(w_1)>p$, a contradiction.

Thus $G$ has a cycle of length $n+1$ or $n+2$ containing $V(C)$.
This implies that $G$ will also have a Hamilton cycle,
which completes the proof of the theorem.
\end{proof}

We continue with two results concerning the graphs in $\mathcal{G}(k)$.

\begin{proposition}
\label{loc:connected}
For every $k\geq 30$,
the set $\mathcal{G}(k)$ contains an infinite class of locally connected graphs
that are not claw-free.
\end{proposition}
\begin{proof}
Let $k$ be an integer, $k\geq 30$. Then $k=5p+t$ for some $p\geq 6$ and $0\leq t\leq 4$.
We will show that for any $n\geq 2$, the set $\mathcal{G}(k)$
contains a locally connected $k$-regular graph of diameter~$n$ which is not claw-free.

Consider first a graph $D(n,p)$ which is defined as follows:
its vertex set is $ V_1\cup\dotsb\cup V_{2n}$,
where $V_1,\dotsc,V_{2n}$ are pairwise disjoint sets of cardinality $2p$
and
two vertices in $V_1\cup\dotsb\cup V_{2n}$
are adjacent if and only if
they both belong to $V_1\cup V_{2n}$ or to $V_i\cup V_{i+1}$
for some $i\in\{1,\dotsc,2n-1\}$.
Denote by $G_i$ the subgraph of $D(n,p)$ induced by the set~$V_i$.
Clearly, $G_i$~is a complete graph on $2p$ vertices
and therefore admits a proper edge coloring with $2p-1$ colors, $i=1,\dots,2n$.

Now from the graph $D(n,p)$ where the edges of the subgraphs
$G_1,\dots,G_{2n}$ are properly colored with $2p-1$ colors, delete all edges having the first $p-1-t$ colors
(note that the conditions $p\ge 6$ and $0\leq t\leq 4$ imply that $p-t-1\geq 1$).
It is not difficult to verify that the resulting graph is ($5p+t$)-regular, locally connected and is not claw-free.
Finally, the resulting graph belongs to the set $\mathcal{G}(k)$ because every vertex of this graph has exactly $5p-t-1$ vertices
at distance~2 and this number does not exceed $k=5p+t$.
\end{proof}

\begin{remark}
Using the same argument as above, one can additionally show that \cref{loc:connected} holds for any
$k\in\set{10,15,16,20,21,22,25,26,27,28}$.
\end{remark}

\begin{proposition}
\label{large-claw-free}
For every integer $k=4n$, $n\geq2$, the set $\mathcal{G}(k)$ contains an infinite class of graphs
that satisfy the conditions of \cref{asratian21}
but are not locally connected or claw-free.
\end{proposition}
\begin{proof}
Let $k=4n$, for some $n\geq 2$. In order to prove the proposition, we will show that for any integer $p\geq 2$
there is a graph of diameter $\floor[\big]{\tfrac{3p}2}$
which satisfies the conditions of \cref{asratian21} but is not locally connected or claw-free.

First consider a graph $Q(4n,p)$ which is defined as follows:
\begin{enumerate}
\item
its vertex set is $ V_1\cup\dotsb\cup V_{3p}$,
where $V_1,\dotsc,V_{3p}$ are pairwise disjoint sets of cardinality
$|V_1|=|V_2|=2n$, $|V_3|=2$ and $|V_{3i+1}|=2n$, $|V_{3i+2}|=2n-1$, $|V_{3i+3}|=2$, for
$i=1,\dots,p-1$,
\item
two vertices in $ V_1\cup\dotsb\cup V_{3p}$
are adjacent if and only if
they both belong to $V_1\cup V_{3p}$ or to $V_i\cup V_{i+1}$
for some $i\in\{1,\dotsc,3p-1\}$.
\end{enumerate}
Let $M_j$ be a perfect matching in the subgraph induced by the set~$V_j$, $j=1,2,3$.
Then it is not difficult to verify that the graph
$Q(4n,p)-(M_1\cup M_2\cup M_3)$ is a $4n$-regular graph of diameter~$\floor[\big]{\tfrac{3p}2}$
which satisfies the conditions of \cref{asratian21} but is not locally connected or claw-free.
The resulting graph for $n=2$ and $p=2$
can be seen in \cref{fig:claw-free}.
\end{proof}

\begin{figure}
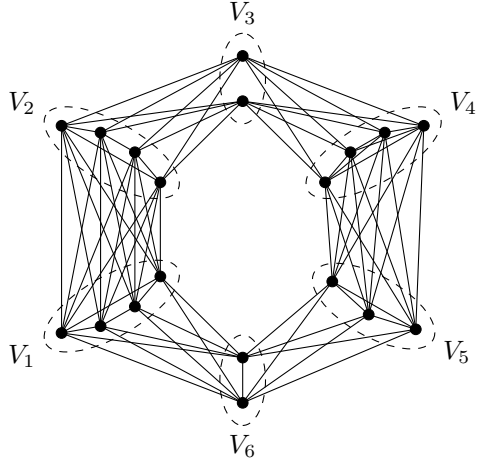

\centering
\figclawfreeexample
\caption{A Hamiltonian graph from $\mathcal{G}(8)$ with 2-connected balls of radius~2 that is not locally connected and not claw-free}
\label{fig:claw-free}
\end{figure}

\Cref{large-claw-free} shows that \cref{asratian21} does not follow
from any known result
on Hamiltonicity of claw-free or locally connected graphs.

The following property was proved in \cite{gordon11}:
\begin{proposition}[Gordon, Orlovich, Potts, and Strusevich~\cite{gordon11}]
\label{gordon11}
Let $G$ be a connected $k$-regular graph with $k\geq 9$ where
every edge belongs to at least $k-4$ triangles. Then $G$ is
\emph{fully cycle extendable}, that is, every vertex of~$G$ belongs to a triangle and for every
non-Hamiltonian cycle~$C$ in~$G$ there exists a cycle~$C'$ of~$G$ such that
$V(C)\subset V(C')$ and $|V(C')|=|V(C)|+1$.
\end{proposition}

We will show now that all graphs satisfying the conditions of \cref{gordon11} belong to the set $\mathcal{G}(k)$.

\begin{proposition}
\label{cycle:extend}
Let $G$ be a connected $k$-regular graph with $k\geq 9$ where every edge
belongs to at least $k-4$ triangles. Then $G\in \mathcal{G}(k)$.
\end{proposition}
% In fact, this is also true for k=8

\begin{proof}
Let $G$ be $k$-regular graph with $k\geq 9$
where every edge belongs to at least $k-4$ triangles.
Consider an arbitrary vertex~$v$ in~$G$. We will show that $|N_2(v)|\leq k$. Let
$e(N_1(v),N_2(v))$ denote the number of edges between $N_1(v)$ and $N_2(v)$.
Clearly, every vertex $u\in N_1(v)$ has at most 3~neighbors in $N_2(v)$, since the edge $uv$
belongs to at least $k-4$ triangles. This implies that $e(N_1(v),N_2(v))\leq 3k$.
Furthermore, for every neighbor~$w$ of~$u$ in~$N_2(v)$
the edge $uw$ belongs to at most $2$~triangles
with the third vertex in $N_2(v)$.
Therefore the vertex~$w$ must have at least $k-6$ neighbors in
$N_1(v)$ different from~$u$, since
the edge $uw$ belongs to at least $k-4$ triangles.
This implies that
$e(N_1(v),N_2(v))\geq (k-5)|N_2(v)|$. Then
$$3k\geq e(N_1(v),N_2(v))\geq (k-5)|N_2(v)|$$
which implies that $|N_2(v)|\leq k$ if $k\geq 9$.
Thus $G\in \mathcal{G}(k)$.
\end{proof}

Taking into consideration \cref{like:bondy}, \cref{gordon11} and \cref{cycle:extend}, we believe that
the following conjecture is true:
\begin{conjecture}
Let $G$ be a locally connected graph in $\mathcal{G}(k)$, $k\geq 9$.
Then $G$ is fully cycle extendable.
\end{conjecture}

\section*{Acknowledgment}
The authors thank Carl Johan Casselgren for helpful suggestions on this manu\-script.

\end{document}